\documentclass{article}
\usepackage{cite}

\usepackage{makeidx}
\makeindex
\usepackage[colorlinks,citecolor=blue,linkcolor=blue,urlcolor=blue,filecolor=blue,breaklinks]{hyperref}
\usepackage{pdfpages}

\usepackage[utf8]{inputenc}
\usepackage{amssymb}
\usepackage{amsfonts}
\usepackage{amsmath}
\usepackage{amsthm}
\usepackage{mathtools}
\usepackage{mathrsfs}
\usepackage{cleveref}
\usepackage{relsize}
\usepackage{tikz-cd}
\usepackage{geometry}
\usepackage{graphicx}
\usepackage{booktabs}
\usepackage{array}
\usepackage{paralist}
\usepackage{verbatim}
\usepackage{subfig}
\usepackage{epsfig}
\usepackage{changepage}

\usepackage{fancyhdr}
\usepackage{sectsty}
\allsectionsfont{\sffamily\mdseries\upshape}
\usepackage[nottoc,notlof,notlot]{tocbibind}
\usepackage[titles,subfigure]{tocloft}

\definecolor{lightblue}{rgb}{0.8,0.8,1}

\theoremstyle{plain}
\newtheorem{theorem}{Theorem}

\newtheorem{proposition}[theorem]{Proposition}

\theoremstyle{definition}
\newtheorem{definition}[theorem]{Definition}
\newtheorem{remark}[theorem]{Remark}

\newtheoremstyle{named}%
    {}{}{\itshape}{}{\bfseries}{.}{.5em}{\thmnote{#3}}
\theoremstyle{named}

\newcommand{\R}{\mathbb{R}}
\newcommand{\Z}{\mathbb{Z}}

\newcommand{\id}{\mathrm{Id}}

\renewenvironment{proof}[1][Proof]{\noindent\textbf{#1.} }{\ \rule{0.5em}{0.5em}\par\addvspace{\baselineskip}}

\newcommand{\Heis}{\mathcal{H}}

\renewcommand{\geq}{\geqslant}
\renewcommand{\leq}{\leqslant}

\definecolor{cycleBM}{RGB}{0,150,0}
\definecolor{cycleD}{RGB}{255,0,0}

\begin{document}
\title{Heisenberg homology of ribbon graphs}
\author{Christian Blanchet\footnote{Universit{\'e} Paris Cité and Sorbonne Universit{\'e}, CNRS, IMJ-PRG, F-75006 Paris, France}}

\maketitle
\begin{abstract} 
We review Heisenberg homology of configurations in once bounded surfaces and extend the construction to the regular thickening of a finite graph with ribbon structure.
\end{abstract}


\section*{Introduction}
We defined and studied  in \cite{HeisenbergHomology} the homology of the configuration space of unordered  points in an oriented surface with one boundary component, over local coefficients defined from a representation of the Heisenberg group. It carries a twisted action of the Mapping Class Group, which can be untwisted in good cases. Here we review the construction and extend it to oriented compact surfaces with positive number of boundary components. We then consider the surface associated to a graph with ribbon structure and show that the Heisenberg homology can be extracted from the graph. 
T\^ete \`a t\^ete graphs and twists were introduced by Norbert A'Campo as a combinatorial tool for describing
 mapping classes on surfaces with boundary arising as monodromy of curve singularities. A careful exposition including  a version based on relative ribbon graphs can be found in \cite{AcampoetcAIF}, see also  \cite{Graf}. We expect further computations for the action of t\^ete à t\^ete twists on Heisenberg homologies, with applications to curve singularities. We thank the anonymous referees for careful reading.

\section{Heisenberg homology of surface configurations}\label{HeisSection}
Let $\Sigma_{g,m}$, $g\geq0$, $m\geq 1$, be an oriented surface of  genus $g$ with $m$ boundary components. For $n\geq 2$, the unordered configuration space\index{configuration space} of $n$ points in $\Sigma_{g,m}$ is
\[
\mathcal{C}_{n}(\Sigma_{g,m} )= \{ \{c_{1},\dots,c_{n}\} \subset \Sigma_{g,m} \mid c_i\neq c_j \text{ for $i\neq j$}\}.
\]
 The surface braid group\index{surface braid group}  is then defined as $\mathbb{B}_{n}(\Sigma_{g,m})=\pi_{1}(\mathcal{C}_{n}(\Sigma_{g,m}),*)$. 
A presentation for this group was first obtained by G. P. Scott \cite{Scott} and revisited by Gonz\'ales-Meneses \cite{Gonzalez}, Bellingeri \cite {Bellingeri}.
We fix based loops, $\alpha_1,\dots,\alpha_g,\beta_1,\dots,\beta_g,\gamma_1,\dots,\gamma_{m-1}$, as depicted in Figure \ref{modelSurf}.
The base point $*_1$ belongs to the base configuration $*$. We will use the same notation $\alpha_r$, $\beta_s$, $\gamma_t$, for the corresponding braids where only the first point is moving.
\begin{figure}
\centering
\includegraphics[scale=0.6]{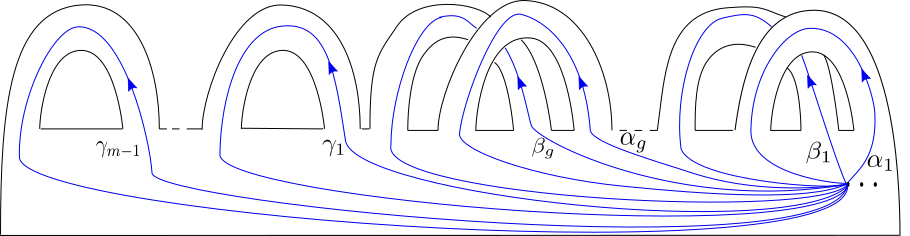}
\caption{Model for $\Sigma_{g,m}$.}
\label{modelSurf}
\end{figure}
The braid group $\mathbb{B}_{n}(\Sigma_{g,m})$ has generators $\alpha_1,\ldots,\alpha_g$, $\beta_1,\ldots,\beta_g$, $\gamma_1,\ldots,\gamma_{m-1}$ together with the classical generators $\sigma_1,\ldots,\sigma_{n-1}$,  and relations:
\begin{equation}
\label{eq:relations}
\begin{cases}
\,\text{(\textbf{BR1}) }\, [\sigma_{i},\sigma_{j}] = 1 & \text{for } \lvert i-j \rvert \geq 2, \\
\,\text{(\textbf{BR2}) }\, \sigma_{i}\sigma_{j}\sigma_{i}=\sigma_{j}\sigma_{i}\sigma_{j} & \text{for } \lvert i-j \rvert = 1, \\
\,\text{(\textbf{CR1}) }\, [\zeta,\sigma_{i}] = 1 & \text{for } i>1 \text{ and all } \zeta\text{ among the }\alpha_r,\beta_s,\gamma_t, \\
\,\text{(\textbf{CR2}) }\, [\zeta,\sigma_{1}\zeta\sigma_{1}] = 1 & \text{for all } \zeta\text{ among the }\alpha_r,\beta_s,\gamma_t, \\
\,\text{(\textbf{CR3}) }\, [\zeta,\sigma^{-1}_{1}\eta\sigma_{1}] = 1& \text{for all } \zeta\neq \eta \text{ among the }\alpha_r,\beta_s,\gamma_t, \text{ with}\\
& \{\zeta,\eta\}\neq \{\alpha_r,\beta_r\},  \text{\color{blue} $\zeta$ on the left of $\eta$},\\
\,\text{(\textbf{SCR}) }\, \sigma_{1}\beta_{r}\sigma_{1}\alpha_{r}\sigma_{1}=\alpha_{r}\sigma_{1}\beta_{r} & \text{for all } r.
\end{cases}
\end{equation}
Composition of loops is written from right to left. Relation (\textbf{CR3}) is equivalent to the formulation in  \cite{BellingeriGodelle}, by using (\textbf{CR2}).

  We will use the notation $x.y$ for the standard intersection form on $H_1(\Sigma_{g,m},\Z)$. The Heisenberg group\index{Heisenberg group}  $\mathcal{H}(\Sigma_{g,m})$ is the central extension of the homology group $H_1(\Sigma_{g,m},\Z)$ defined with the $2$-cocycle $(x,y)\mapsto x.y$. As a set 
 $\mathcal{H}(\Sigma_{g,m})$ is equal to
  $\Z \times H_1(\Sigma_{g,m},\Z)$, and  the operation is given by
\begin{equation}
\label{eq:Heisenberg-product}
(k,x)(l,y)=(k+l+\,x.y,x+y).
\end{equation}
 
 We will use the notation $a_r$, $b_s$, $c_t$ for the homology classes of $\alpha_r$, $\beta_s$,
 $\gamma_t$.
\begin{proposition}
\label{hom_phi}
 For each $g\geq 0$, $m\geq 1$ and $n\geq  2$, the quotient of the braid group $\mathbb{B}_n(\Sigma_{g,m})$ by the subgroup  $[\sigma_1,\mathbb{B}_n(\Sigma_{g,m})]^N$ normally generated by the commutators $[\sigma_1,x]$, $x\in \mathbb{B}_n(\Sigma_{g,m})$, is isomorphic 
to the Heisenberg group $\Heis(\Sigma_{g,m})$. An isomorphism $$\mathbb{B}_n(\Sigma_{g,m})/[\sigma_1,\mathbb{B}_n(\Sigma_{g,m})]^N\approx \mathcal{H}(\Sigma_{g,m}) $$ is  represented by the surjective homomorphism
\[
\phi \colon \mathbb{B}_{n}(\Sigma_{g,m}) \relbar\joinrel\rightarrow \Heis(\Sigma_{g,m})
\]
sending each $\sigma_i$ to $u=(1,0)$, $\alpha_r$ to $\tilde{a}_r=(0,a_r)$, $\beta_s$ to $\tilde{b}_s=(0,b_s)$, $\gamma_t$ to $\tilde{c}_t=(0,c_t)$. \\
\end{proposition}
This is proved in \cite{HeisenbergHomology} for a surface with one boundary component. 

\begin{proof}
We first show that the homomorphism $\phi$ is defined and surjective. 
For proving the existence of $\phi$ it is enough to show that all relations  are satisfied in $\mathcal{H}(\Sigma_{g,m})$ for the images of the generators. This is straightforward for relations  (BR$\ast$) and (CR$\ast$). For (SCR), we get
$$\phi(\sigma_1\beta_r\sigma_1\alpha_r\sigma_1)=(2,a_r+b_r)=\phi(\alpha_r\sigma_1\beta_r).$$
From the formulas, we see that 
$\phi$ is surjective.

We now use the presentation for proving that the quotient map is an isomorphism.
From \eqref{eq:relations} we obtain that the quotient of $\mathbb{B}(\Sigma_{g,m})$ by 
the normal subgroup $[\sigma_1,\mathbb{B}_n(\Sigma_{g,m})]^N$ is generated by $\sigma_1$, $\alpha_1,\dots,\alpha_g,\beta_1,\dots,\beta_g$,  $\gamma_1,\dots,\gamma_{m-1}$,
with relations
\[
\begin{cases}
\sigma_1 \text{ is central},\\
\,\text{(\textbf{CR3}) }\, \zeta\eta = \eta\zeta & \text{for all } \zeta,\eta \text{ among the }\alpha_r,\beta_s,\gamma_t, \text{ with}\\
& \{\zeta,\eta\}\neq \{\alpha_r,\beta_r\},\\
\,\text{(\textbf{SCR}) }\, \sigma_{1}^2\beta_{r}\alpha_{r}=\alpha_{r}\beta_{r} & \text{for all } r.
\end{cases}
\]
This matches a presentation of $\Heis(\Sigma_{g,m})$.
\end{proof}
Using the homomorphism $\phi$ we define a regular covering $\widetilde{\mathcal{C}}_n(\Sigma_{g,m})$ of the configuration space ${\mathcal{C}}_n(\Sigma_{g,m})$. 
The homology of this cover  is what we call 
 the {\em Heisenberg homology}.
Let us denote by $S_*(\widetilde{\mathcal{C}}_n(\Sigma_g))$ the singular  chain complex of the Heisenberg cover  $\widetilde{\mathcal{C}}_n(\Sigma_{g,m})$ which is a right $\Z[\Heis(\Sigma_{g,m})]$-module.
  Given a representation $\rho: \Heis(\Sigma_{g,m}) \rightarrow GL(W)$, the corresponding local homology is that of the complex
   \begin{equation} \label{eq:Local}
   S_*(\mathcal{C}_n(\Sigma_{g,m});W):=S_*(\widetilde{\mathcal{C}}_n(\Sigma_{g,m}))\otimes_{\Z[\Heis(\Sigma_{g,m})]} W
   \end{equation}
It will be called the Heisenberg homology of surface configurations with coefficients in $W$.

It is convenient  to also consider  Borel-Moore homology\index{Borel-Moore homology}
\begin{equation}\label{eq:BorelMoore}
H_*^{BM}(\mathcal{C}_{n}(\Sigma_{g,m});W) = 
{\varprojlim_T}\, H_*(\mathcal{C}_{n}(\Sigma_{g,m}), \mathcal{C}_{n}(\Sigma_{g,m})\setminus T ; W),
\end{equation}
the inverse limit is taken over all compact subsets  $T\subset\mathcal{C}_{n}(\Sigma_{g,m})$. The relative homology with twisted coefficient $H_*(\mathcal{C}_{n}(\Sigma_{g,m}), \mathcal{C}_{n}(\Sigma_{g,m})\setminus T ; W)$
is obtained from the quotient of the complex $S_*(\mathcal{C}_n(\Sigma_{g,m});W)$ by the subcomplex 
$ S_*(\widetilde{\mathcal{C}}_n(\Sigma_{g,m})\setminus \widetilde T)\otimes_{\Z[\Heis(\Sigma_{g,m})] }W$, where $\widetilde T$ is the inverse image of $T$ in the covering space $\widetilde{\mathcal{C}}_n(\Sigma_{g,m})$.

\section{Ribbon graphs}
A ribbon structure\index{ribbon structure} on a finite graph  is  a cyclic ordering of the adjacent edges at each vertex. Given a graph with ribbon structure $\Gamma$, there is a regular thickening $\Sigma_\Gamma$ which is a compact surface containing $\Gamma$ as a strong deformation retract.
This surface $\Sigma_\Gamma$ is usually called a ribbon graph\index{ribbon graph} or fat graph.
For $n\geq 2$, we may consider the unordered configuration space $\mathcal{C}_{n}(\Gamma)\subset \mathcal{C}_{n}(\Sigma_\Gamma)$ and the graph braid group $\mathbb{B}_{n}(\Gamma)=\pi_1(\mathcal{C}_{n}(\Gamma),*)$.
 Graph braid groups and homology of graph configurations are intensively studied \cite{Abrams,CrispWiest,FarleySabalka,Kurlin,Chettih,MaciazekSawicki,Anetc,AnMaciazek}. It will be good to connect computation for surface configurations and for graph configurations.
 Following a so called {\em compression trick} \cite{Bigelow2001,Martel,HeisenbergHomology} we can prove the following.
 
 \begin{theorem}\label{ThmIso}
 The inclusion $\mathcal{C}_{n}(\Gamma)\subset \mathcal{C}_{n}(\Sigma_\Gamma)$ induces an isomorphism on Borel-Moore homologies with local coefficients defined by a representation $\rho: \Heis(\Sigma_\Gamma)\rightarrow GL(W)$,
\begin{equation}\label{eq:ThmIso} 
 H_*^{BM}(\mathcal{C}_{n}(\Gamma);W)\cong H_*^{BM}(\mathcal{C}_{n}(\Sigma_\Gamma);W)\ .
 \end{equation}
 \end{theorem}
Here the local system on  $\mathcal{C}_{n}(\Sigma_\Gamma)$ induces a local system on the subspace $\mathcal{C}_{n}(\Gamma)$. The corresponding Borel-Moore homology is computed from the subcomplex 
$S_*({\mathcal{C}}_n(\Gamma);W)\subset S_*({\mathcal{C}}_n(\Sigma_\Gamma);W)$ in a similar way to \eqref{eq:BorelMoore}, by an analogous inverse limit of quotients of this subcomplex. Note that the homotopy type of the configuration space of a graph does not depend on the ribbon structure, but the local system induced from the thickening does depend.

\begin{proof}
We put a metric on the graph $\Gamma$; the choice is irrelevant but we find it convenient to think that all edges are linear with length $l$ big enough.
The surface $\Sigma_\Gamma$ can be obtained from the graph $\Gamma$ as follows. For each edge $e$ in $\Gamma$ we take a band $[-1, 1]\times e$ equipped with product metric. For each vertex we glue the half ends of the bands according to the cyclic ordering. We obtain the surface $\Sigma_\Gamma$ with the deformation retraction $h_t$, $0\leq t\leq 1$, which sends $(s,x)\in [-1,1]\times e$ to $((1-t)s,x)$. Let us use the metric $d$ on $\Sigma_\Gamma$ defined by the shortest path, where the length of a path is the sum of the length of its intersections with the bands. 
Then we have that  $h_t$ is  $1$-Lipschitz for $0\leq t\leq 1$ and is an embedding for $0\leq t< 1$. The difficulty in the proof is that $h_1$ is not injective so that $\mathcal{C}_n(h_1)$ is only defined outside the collision subset
$$Z=\{c=\{c_1,\dots,c_n\}\subset \Sigma_\Gamma\text{ such that }{h_1(c_i)=h_1(c_j)}\text{ for some $i\neq j$}\}\ ,$$
This will be carried over with an excision argument. 

For $0\leq t<1$, let $\Sigma_t=h_t(\Sigma_\Gamma)$ and 
 $Z_t=Z\cap \mathcal{C}_n(\Sigma_t)$.
We quote that we are using a very specific deformation retraction    for which we have  $Z_t=\mathcal{C}_n(h_t)(Z)$ for $0\leq t <1$.
For $\epsilon > 0$ and $Y\subset \Sigma_\Gamma$ we denote by $\mathcal{C}_n^\epsilon(Y)$   the subspace of configurations \mbox{${c}=\{ c_1,c_2,...,c_n\}\subset Y$} such that $d(c_i , c_j) <\epsilon$ for some $i\ne j$ 
and by $\mathcal{C}_n^{\geq \epsilon}(Y)$ its complement in $\mathcal{C}_n(Y)$. If $Y$ is closed, then $\mathcal{C}_n^{\geq \epsilon}(Y)$ is compact and 
any compact subset $T\subset \mathcal{C}_n(Y)$ is included is some   $\mathcal{C}_n^{\geq \epsilon}(Y)$, which implies
\begin{equation}
\label{eq:BM-as-limit}
H^{BM}_{*}(\mathcal{C}_{n}(Y);W) \;\cong\; \lim_{0\leftarrow \epsilon}H_{*}(\mathcal{C}_{n}(Y), \mathcal{C}_n^\epsilon(Y);W)
\end{equation}

  For $t<1$ the inclusion map
\[
(\mathcal{C}_{n}(\Sigma_t), \mathcal{C}_n^\epsilon(\Sigma_t))\,\subset\,(\mathcal{C}_{n}(\Sigma_\Gamma), \mathcal{C}_n^\epsilon(\Sigma_\Gamma))
\]
is a homotopy equivalence with homotopy inverse $\mathcal{C}_n(h_t)$, which is a map of pairs homotopic to the identity because $h_\nu$ is $1$-Lipschitz for $0\leq \nu\leq t$. So we have an inclusion isomorphism
\begin{equation}
\label{eq:iso1}
H_*(\mathcal{C}_{n}(\Sigma_t), \mathcal{C}_n^\epsilon(\Sigma_t);W)\;\cong\; H_*(\mathcal{C}_{n}(\Sigma_\Gamma), \mathcal{C}_n^\epsilon(\Sigma_\Gamma);W)
\end{equation}
For $\epsilon>0$ we may choose $t=t_\epsilon <1$ such that for all $p \in \Sigma_\Gamma$ we have \mbox{$ d(h_t(p) , h_1(p)) < \tfrac{\epsilon}{2}$}. 
For such $t$, we have that the collision subset $Z_t=Z\cap \mathcal{C}_n(\Sigma_t)$, which is closed,  is contained in the open set $ \mathcal{C}_n^\epsilon(\Sigma_t) $. We therefore get an excision isomorphism
\begin{equation}
\label{eq:iso2}
H_*(\mathcal{C}_{n}(\Sigma_t)\setminus Z_t ,\mathcal{C}_n^\epsilon(\Sigma_t) \setminus Z_t;W) \;\cong\; H_*(\mathcal{C}_{n}(\Sigma_t), \mathcal{C}_n^\epsilon(\Sigma_t);W)
\end{equation}
The inclusion map $$(\mathcal{C}_{n}(\Gamma), \mathcal{C}_n^\epsilon(\Gamma)) \, \longrightarrow \, (\mathcal{C}_{n}(\Sigma_t)\setminus Z_t, \mathcal{C}_n^\epsilon(\Sigma_t)\setminus Z_t)\ ,$$
has an homotopy inverse represented by the restriction of $\mathcal{C}_{n}(h_1)$. An homotopy between the restriction of $\mathcal{C}_{n}(h_1)$ and the identity of $\mathcal{C}_n(\Sigma_t)\setminus Z_t$ is given by the path of restrictions of $\mathcal{C}_{n}(h_\nu)$, $0\leq \nu \leq 1$. Here we use that $h_\nu$ is a contraction and that 
$$\mathcal{C}_{n}(h_\nu)(\mathcal{C}_n(\Sigma_t)\setminus Z_t)=\mathcal{C}_n(\Sigma_{t(1-\nu)})\setminus Z_{t(1-\nu)}
\subset\mathcal{C}_n(\Sigma_t)\setminus Z_t\ ,\ 0\leq s<1\ .$$
 We obtain the isomorphism
\begin{equation}
\label{eq:iso3}
H_{*}(\mathcal{C}_{n}(\Gamma),\mathcal{C}_n^{\epsilon}(\Gamma);W) \;\cong\;
\,H_*(\mathcal{C}_{n}(\Sigma_t)\setminus Z_t ,\mathcal{C}_n^\epsilon(\Sigma_t) \setminus Z_t;W)
\end{equation}
By composing the isomorphisms \eqref{eq:iso3}, \eqref{eq:iso2} and \eqref{eq:iso1}, we obtain the inclusion isomorphism:
\begin{equation}
\label{eq:iso4}
H_{*}(\mathcal{C}_{n}(\Gamma),\mathcal{C}_n^{\epsilon}(\Gamma);W) \;\cong\;
\,H_{*}(\mathcal{C}_{n}(\Sigma_\Gamma), \mathcal{C}_n^{\epsilon}(\Sigma_\Gamma);W).
\end{equation}
The limit $\epsilon \rightarrow 0$ gives the expected  isomorphism \eqref{eq:ThmIso} for Borel-Moore homology.
\end{proof}

\subsection*{Relative ribbon graphs}
T\^ete à t\^ete graphs and twists have been introduced by Norbert A'Campo and extended to a version based on relative ribbon graphs  \cite[Definition 2.2]{AcampoetcAIF}. 
Here we extend the definition of relative ribbon graphs to the case where the components of the subgraph $A$ can be circles or intervals. We will be mostly interested with the case where $A$ is a single interval.

\begin{definition}
A finite graph with relative ribbon structure\index{ribbon structure}  is a finite graph $\Gamma$, a subgraph whose components $A_i$ are oriented circles or intervals, and a cyclic ordering  of the adjacent edges at each vertex, where we ask that the cyclic ordering at a vertex $v$ in $A_i$ is compatible with the orientation of $A_i$.  Here the compatibility condition asks that the cyclic ordering at a vertex in $A$ is represented by a sequence of edges $(e_1,\dots,e_k)$, where $e_1$ is the incoming edge from $A_i$ if any, $e_k$ is the outgoing edge in $A_i$ if any, and all other edges are not in $A$. Here we put  {\em if any} in view that at an initial (resp. terminal) vertex of an interval $A_i$ there is no incoming (resp. outgoing) edge.
\end{definition}
There is a regular thickening $\Sigma_{(\Gamma,A)}$ of a graph with relative ribbon structure. For each edge $e$ in $\Gamma\setminus A$ we take a band $[-1, 1]\times e$, and for an edge $e$ in $A$ a {\em half band} $[0,1]\times e$. For each vertex we glue the half ends of the bands or ends of half bands according to the cyclic ordering.  The embedded graph $\Gamma \subset \Sigma_{(\Gamma,A)}$ is the image in the gluing of $\cup_e \{0\}\times e$. Each $A_i=\{0\}\times A_i$ is a boundary component or an interval in a boundary component with orientation opposite to $\partial \Sigma_\Gamma$.
An example is represented in Figure \ref{modelG}.
\begin{figure}
\centering
\includegraphics[scale=0.6]{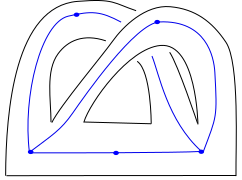}\hspace{2cm}\includegraphics[scale=0.6]{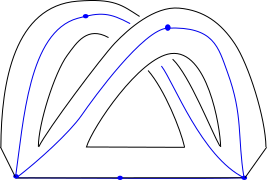}
\caption{Surfaces $\Sigma_\Gamma$ and $\Sigma_{(\Gamma,A)}$ for a relative graph $(\Gamma,A)$, where $A$ is the horizontal interval.}
\label{modelG}
\end{figure}
We have a strong deformation retraction from $(\Sigma_{(\Gamma,A)},A)$ to $(\Gamma,A)$. 
The compression trick used in Theorem \ref{ThmIso} is valid for a relative ribbon graph and we can prove the extended result below. For a pair of spaces $Y\subset X$, we denote by $\mathcal{C}_n(X,Y)\subset \mathcal{C}_n(X)$ the subspace of configurations with at least one point in $Y$.
 \begin{theorem}\label{ThmIso2}
 The inclusion $\mathcal{C}_{n}(\Gamma)\subset \mathcal{C}_{n}(\Sigma_{(\Gamma,A)})$ induces isomorphisms on Borel-Moore homologies with local coefficients defined by a representation $\rho: \Heis(\Sigma_{(\Gamma,A)})\rightarrow GL(W)$.
 \begin{equation}\label{eq:ThmIsor} 
 H_*^{BM}(\mathcal{C}_{n}(\Gamma);W)\cong H_*^{BM}(\mathcal{C}_{n}(\Sigma_{(\Gamma,A)});W)\ .
 \end{equation}
 \begin{equation}\label{eq:ThmIsorr} 
 H_*^{BM}(\mathcal{C}_{n}(\Gamma),\mathcal{C}_{n}(\Gamma,A)  ;W)\cong H_*^{BM}(\mathcal{C}_{n}(\Sigma_{(\Gamma,A)}),\mathcal{C}_n(\Sigma_{(\Gamma,A)},A);W)\ .
 \end{equation}
 \end{theorem}
\section{A complex for Borel-Moore homology of graph configurations}
Configurations of points in a graph have been intensively studied. It is proved in \cite{Abrams,PrueScrimshaw,FarleySabalka} that a discrete version of the graph configuration space, where cells at infinity are removed, is homotopy equivalent to the full configuration space when the graph is enough subdivided, which allows to give a finite cell decomposition. Presentations for the fundamental group (the graph braid group) and the homology follow. The presentation of homology can be significantly reduced using discrete Morse theory.

In view of Borel-Moore homology, we may work with a finite cell decomposition which does not confer a CW-complex structure, but nevertherless gives a skeleton filtration which works similarly for computing Borel-Moore homology.

Let $\Gamma$ be a finite graph with set of vertices $V$ and set of edges  $E$. Let $n\geq 2$, and suppose that $\sharp V\geq n$, then we set $\mathcal{C}_n^0(\Gamma)=\mathcal{C}_n(V)$ and for $0<k\leq n$, 
$\mathcal{C}_n^k(\Gamma)$ is the union of $\mathcal{C}_n^{k-1}(\Gamma)$ and the $k$-dimensional cells 
$$\prod_{x\in V_0} x \times \prod_{e\in E} \mathcal{C}_{k_e}(e)\ , \ V_0\subset V\ , \ \sharp V_0=k_0\ , \ k_0+\sum_e k_e=k_0+k=n\ .$$

In the case $n=2$, the $0$-cells are pairs in $V$, the $1$-cells are $v\times e$, $v\in V$, $e\in E$, the $2$-cells are either squares $e\times f$, $e\neq f$, or simplices  $\mathcal{C}_{2}(e)$. The $k$-cells are attached to $\mathcal{C}_n^{k-1}(\Gamma)$ along the faces which are not diagonal. Each $k$-cell is properly embedded in $\mathcal{C}_n^{k}(\Gamma)$ and will contribute to a $\Z$ summand in the homology $H_k^{BM}({\mathcal{C}}_n^{k}(\Gamma),{\mathcal{C}}_{n}^{k-1}(\Gamma))$. If the graph $\Gamma$ has a ribbon structure, then  we have a corresponding filtration of the Heisenberg cover $\widetilde{\mathcal{C}}_n^k(\Gamma)$. Then a lift of a $k$-cell generates
a  $\Z[\Heis]$ summand in the homology $H_k^{BM}(\widetilde{\mathcal{C}}_n^{k}(\Gamma),\widetilde{\mathcal{C}}_{n}^{k-1}(\Gamma))$ and a $W$ summand in the homology $H_k^{BM}({\mathcal{C}}_n^{k}(\Gamma),{\mathcal{C}}_{n}^{k-1}(\Gamma);W)$.

If we fix an orientation of $\Gamma$, then the cells are oriented and we have an incidence sign $[E^k;F^{k-1}]$ for an off diagonal face $F^{k-1}$ of a $k$-cell $E^k$ which is a product of open simplices.
The Borel-Moore cellular chain complex\index{Borel-Moore cellular chain complex} of $\mathcal{C}_n(\Gamma)$ 
is the graded module $C_*^{BM}(\mathcal{C}_n(\Gamma))=\oplus_k C_k^{BM}(\mathcal{C}_n(\Gamma))$ freely generated in degree $k$ by the $k$-cells, with boundary map given by incidence numbers.
If the oriented edge $e$ has boundary $w-v$ in the graph $\Gamma$, then the boundary of the edge $u\times e$, $u\in V$, $e\in E$,
is \begin{equation}
\partial(u\times e)=\begin{cases}
\{u,w\}-\{u,v\}&\text{ if $u\notin\{v,w\}$}\\
\{u,w\} &\text{ if $u=v$}\\
-\{u,v\}& \text{ if $u=w$}
\end{cases}
\end{equation}

The Borel-Moore cellular chain complex of a regular $G$-covering $\widetilde{\mathcal{C}}_n(\Gamma)$ of  $\mathcal{C}_n(\Gamma)$
is the graded $\Z$-module $C_*^{BM}(\widetilde{\mathcal{C}}_n(\Gamma))$ freely generated in degree $k$ by all lifts of the $k$-cells, with boundary map given by incidence numbers. It is convenient to choose a preferred lift $\tilde E$ for each cell $E$ in $\mathcal{C}_n(\Gamma)$  ; this can be done by using a path or tether from the base configuration  to $E$. Then $C^{BM}_k(\widetilde{\mathcal{C}}_n(\Gamma))$ is freely generated over $\Z[G]$ by the (lifted) $k$-cells in $\mathcal{C}_n(\Gamma)$. Now the  incidence number $[\tilde E^k;\tilde F^{k-1}]$ is the previous sign times an element in $G$ which counts the deck action from the lift  $\tilde F^{k-1}$ to the actual face of the lift  $\tilde E^k$.
Given a representation $\rho: G\rightarrow GL(W)$, we get a complex suitable for local coefficients in $W$
$$C_*^{BM}({\mathcal{C}}_n(\Gamma);W)=C_*^{BM}(\widetilde{\mathcal{C}}_n(\Gamma))\otimes_{\Z[G]} W\ .$$
We plan to give elsewhere a general treatment for Borel-Moore cell complexes. The result for graph configurations is the following.
\begin{theorem}\label{thm:BMcomplex}
Let $\Gamma$ be a finite graph, $n\geq 2$. For a representation $\rho: \mathbb{B}_n(\Gamma)\rightarrow GL(W)$,
the Borel-Moore homology $H_*^{BM}({\mathcal{C}}_n(\Gamma);W)$ is canonically isomorphic to the homology of the Borel-Moore cellular chain complex $C_*^{BM}({\mathcal{C}}_n(\Gamma);W)$.
\end{theorem}
For a relative graph $(\Gamma,A)$ we have a relative Borel-Moore cellular chain complex $C_*^{BM}({\mathcal{C}}_n(\Gamma),{\mathcal{C}}_n(\Gamma,A))$ freely generated in degree $k$ by the $k$-dimensional cells 
$$\prod_{x\in V_0} x\times \prod_{e\in E_{\Gamma\setminus A}} \mathcal{C}_{k_e}(e)\ , \ V_0\subset V_{\Gamma\setminus A}\ , \ \sharp V_0=k_0\ , \ k_0+\sum_e k_e=k_0+k=n\ ,$$
where $V_{\Gamma\setminus A}$ (resp. $E_{\Gamma\setminus A}$) denote the set of vertices (resp. edges) not in $A$.
Proceeding as before, we also have a relative Borel-Moore cellular chain complex with local coefficients
$ C_*^{BM}({\mathcal{C}}_n(\Gamma),\mathcal{C}_n(\Gamma,A) ;W)$ which computes the homology
$ H_*^{BM}({\mathcal{C}}_n(\Gamma),\mathcal{C}_n(\Gamma,A) ;W)$.

A surface $\Sigma_{g,m}$ with an interval $A$ in its boundary can be obtained as the regular thickening of a relative graph $(\Gamma,A)$, where $\Gamma \setminus A$ contains $2g+m-1$ edges $e_1,\dots, e_{2g+m-1}$ going from the end vertex of $A$ to the origin vertex of $A$. In this situation, the relative Borel-Moore complex has only $n$-cells corresponding to the partitions $n=n_1+\dots+n_{2g+m-1}$.
 From Theorems \ref{ThmIso2} and \ref{thm:BMcomplex} we obtain the following result which was proved in the case $m=1$ in  \cite{HeisenbergHomology}.
\begin{theorem}
\label{athm:twisted}
Let $\Sigma_{g,m}$ be a genus $g$ surface with $m>0$ boundary components, and $A\subset \partial \Sigma_{g,m}$ be an interval.
Let $n\geq 2$ and let $W$ be a representation of the discrete Heisenberg group $\Heis = \Heis(\Sigma_{g,m})$ over a ring $R$.\\
  The Borel-Moore homology module $H_n^{BM}(\mathcal{C}_{n}(\Sigma_{g,m})),\mathcal{C}_{n}(\Sigma_{g,m},A) ;W)$ is isomorphic to the direct sum of $\left(
\begin{array}{c}
2g+m+n-2 \\
n \\
\end{array}
\right)$ copies of $W$.
Furthermore, $H_*^{BM}(\mathcal{C}_{n}(\Sigma_{g,m}),\mathcal{C}_{n}(\Sigma_{g,m},A) ;W)$ vanishes for $*\neq n$.
\end{theorem}

\section{Action of mapping classes}
Here we work with a compact oriented surface $\Sigma=\Sigma_{g,m}$ with non empty boundary and fixed  based loops, $\alpha_1,\dots,\alpha_g,\beta_1,\dots,\beta_g,\gamma_1,\dots,\gamma_{m-1}$, as depicted in Figure \ref{modelSurf}. We denote by $\partial^1\Sigma$ a non empty subset of the boundary which is a union of connected components. The $\partial^1\Sigma$-relative
{Mapping Class Group}  $\mathfrak{M}(\Sigma,\partial^1\Sigma)$ is the group of orientation preserving homeomorphisms of $\Sigma$ fixing  pointwise $\partial^1\Sigma$, modulo isotopies relative to $\partial^1\Sigma$. The isotopy class of a homeomorphism $f$ is denoted by $[f]$. An oriented self-homeomorphism  $f \colon \Sigma \rightarrow \Sigma$, fixing  pointwise  $\partial^1\Sigma$, gives a homeomorphism $\mathcal{C}_{n}(f) \colon \mathcal{C}_{n}(\Sigma) \rightarrow \mathcal{C}_{n}(\Sigma)$, defined by $\{c_{1},c_{2},\ldots ,c_{n}\} \mapsto \{f(c_{1}),f(c_{2}),\ldots ,f(c_{n})\}$. We choose the base configuration in $\partial^1\Sigma$ so that it is fixed by $\mathcal{C}_n(f)$, then 
$\mathcal{C}_{n}(f)$
induces a homomorphism $f_{\mathbb{B}_{n}(\Sigma)}  \colon \mathbb{B}_{n}(\Sigma) \rightarrow \mathbb{B}_{n}(\Sigma)$, which depends only on the isotopy class $[f]$. 
 
\subsection{Action on the Heisenberg group}
We first study the induced action on the Heisenberg group quotient $\Heis(\Sigma)$. 
From Proposition \ref{hom_phi} we have a quotient map $\phi:  \mathbb{B}_{n}(\Sigma)\rightarrow \Heis(\Sigma)$.
We denote by $\mathrm{Aut}^+(\mathcal{\Heis}(\Sigma))$ the group of automorphisms of $\Heis(\Sigma)$ which are identity on $\phi(\sigma_1)$. 
\begin{proposition}
\label{f_Heisenberg}
The functorial automorphism $\mathbb{B}_n(f):  \mathbb{B}_{n}(\Sigma)\rightarrow \mathbb{B}_{n}(\Sigma)$ induces an action 
\begin{equation}\label{eq:action_on_Heis}
\begin{array}{lrll}
\Psi \colon & \mathfrak{M}(\Sigma,\partial^1\Sigma)&\rightarrow &\mathrm{Aut}^+(\Heis(\Sigma))\\
&f &\mapsto &f_\Heis 
\end{array}
\end{equation}
\end{proposition}
\begin{proof}
Up to isotopy we may suppose that $f\in\mathfrak{M}(\Sigma,\partial^1\Sigma) $ is identity in a neighbourhood of the base configuration. It follows that $f_{\mathbb{B}_{n}(\Sigma)}$ fixes $\sigma_1$. We deduce from Proposition \ref{hom_phi} that $f_{\mathbb{B}_{n}(\Sigma)}$ sends $\ker(\phi)$ to itself. 
It follows that there exists a unique automorphism $f_{\Heis}\in \mathrm{Aut}^+(\Heis(\Sigma))$ such that the following square commutes
\begin{equation}
\label{eq:projection-equivariance}
  \begin{tikzcd}
     \mathbb{B}_{n}(\Sigma) \arrow[d,swap, "\phi"]\arrow[r, "f_{\mathbb{B}_{n}(\Sigma)}"] & \mathbb{B}_{n}(\Sigma) \arrow[d,"\phi"]\\
     \Heis(\Sigma) \arrow[r, "f_{\Heis}"]& \Heis(\Sigma)
  \end{tikzcd}
\end{equation}
 which completes the proof.
\end{proof}

\subsection{Twisted action of the Mapping Class Group}
We fix the base configuration in $\partial^1\Sigma$.
For a mapping class in $\mathfrak{M}(\Sigma, \partial^1\Sigma)$, we take a representative $f$ and denote by $\mathcal{C}_n(f)$ the corresponding homeomorphism of the configuration space $\mathcal{C}_n(\Sigma)$.
For a representation $\rho: \Heis(\Sigma)\rightarrow GL(W)$ and $\tau\in \mathrm{Aut}^+(\Heis(\Sigma))$, we denote by
${}_{\tau}\!W $ the twisted representation $\rho\circ \tau$.

\begin{theorem}
There is a natural twisted representation of the mapping class group $\mathfrak{M}(\Sigma,\partial^1\Sigma)$ on 
\begin{equation*}
H_*\bigl( \mathcal{C}_n(\Sigma_g),{}_\tau\!W\bigr) \ , \quad \tau \in \mathrm{Aut}^+(\Heis(\Sigma))\ ,\ 
\end{equation*}
where the action of $f\in \mathfrak{M}(\Sigma,\partial^1\Sigma)$ is
\begin{equation*}
\label{twisted}
\mathcal{C}_n(f)_* \colon H_*\bigl( \mathcal{C}_n(\Sigma_g) , {}_{\tau\circ f_\Heis}\!W\bigr) \longrightarrow H_*\bigl( \mathcal{C}_n(\Sigma) ,{}_{\tau}\!W \bigr)
\end{equation*}
\end{theorem}
There are similar twisted representations on the Borel-Moore homologies
$H_*^{BM}\bigl( \mathcal{C}_n(\Sigma);{}_\tau\!W \bigr)$ or on the relative versions 
$H_*\bigl( \mathcal{C}_n(\Sigma), \mathcal{C}_n(\Sigma,A);{}_\tau\!W \bigr)$, $H_*^{BM}\bigl( \mathcal{C}_n(\Sigma), \mathcal{C}_n(\Sigma,A);{}_\tau\!W \bigr)$ for $A$ an interval in  $\partial^1\Sigma$.
The above theorem is proved in \cite{HeisenbergHomology} in the case of a surface with one boundary component. We give a sketch proof below.

\begin{proof}
Recall that the homology with coefficients in $W$ is computed from the complex $S_*(\mathcal{C}_n(\Sigma);W):=S_*(\widetilde{\mathcal{C}}_n(\Sigma))\otimes_{\Z[\Heis(\Sigma)]} W$.
The action of $\mathcal{C}_n(f)$ on $S_*(\widetilde{\mathcal{C}}_n(\Sigma))$ is twisted with respect to $\Z[\Heis(\Sigma)]$ action, which writes down
$$S_*(\widetilde{\mathcal{C}}_n(f))(zh)=S_*(\widetilde{\mathcal{C}}_n(f)(z)f_\Heis(h), \text{ for $z\in S_*(\widetilde{\mathcal{C}}_n(\Sigma_g))$, $h\in \Heis(\Sigma)$.}$$
We check that the map $z\otimes v\mapsto S_*(\widetilde{\mathcal{C}}_n(f))(z)\otimes v$ defines an isomorphism
$$S_*(\mathcal{C}_n(\Sigma),{}_{\tau\circ f_\Heis}\!W)\rightarrow S_*(\mathcal{C}_n(\Sigma),{}_{\tau}\!W)\ ,$$
which induces the functorial twisted action on the homologies.
\end{proof}
\subsection{MCG representations from the regular action on Heisenberg group}\label{regular}
We gave in \cite{HeisenbergHomology,HeisenbergClosed} examples of representations where the above twisted representation of the Mapping Class Group can be redesigned into a native representation. We briefly present here the linearised translation action on the Heisenberg group.

For $(k_0,x_0)\in \Heis(\Sigma)$, the 
 left regular action $l_{(k_0,x_0)}$ is an affine automorphism of $\Heis(\Sigma)\approx \Z\oplus H_1(\Sigma,\Z)\approx \Z^{2g+m}$. Here $\Sigma=\Sigma_{g,m}$ has genus $g$ and $m$ boundary components.
We  consider the linearisation $\rho_L$ of this affine action on  $L=\Heis(\Sigma)\oplus \Z \supset \Heis(\Sigma)\oplus 1\cong \Heis(\Sigma)$.
The nice feature of this representation is that the twisted representation ${}_{\tau}\!L$  is canonically isomorphic to $L$ \cite{HeisenbergHomology,HeisenbergClosed}. Indeed, for 
 $\tau \in \mathrm{Aut}^+(\Heis(\Sigma))$,  the linear map $\tau\times \id_\Z: L\mapsto {}_{\tau}\!L$ gives an isomorphism of  
$\Z[\Heis(\Sigma)]$-module.
We then obtain.
\begin{theorem}
a)  There is  a representation
$$\mathfrak{M}(\Sigma,\partial^1\Sigma)\rightarrow \mathrm{Aut}( H_*\bigl( \mathcal{C}_n(\Sigma);L \bigr))\ ,$$
which associates to $f\in  \mathfrak{M}(\Sigma,\partial^1\Sigma )$
the composition of the coefficient isomorphism  induced by $f_\Heis$, $$H_*\bigl( \mathcal{C}_n(\Sigma);L \bigr)\cong
H_*\bigl( \mathcal{C}_n(\Sigma) ,{}_{f_\Heis}\!L \bigr)\ ,$$
 with the functorial homology isomorphism
$$\mathcal{C}_n(f)_* \colon H_*\bigl( \mathcal{C}_n(\Sigma) , {}_{f_\Heis}\!L\bigr) \longrightarrow H_*\bigl( \mathcal{C}_n(\Sigma) ;L \bigr)\ .$$
b) There are similar representations on the Borel-Moore homologies
$H_*^{BM}\bigl( \mathcal{C}_n(\Sigma);L \bigr)$ and on the relative versions 
$H_*\bigl( \mathcal{C}_n(\Sigma), \mathcal{C}_n(\Sigma,A);L \bigr)$, $H_*^{BM}\bigl( \mathcal{C}_n(\Sigma), \mathcal{C}_n(\Sigma,A);L \bigr)$ for $A$ an interval in  $\partial^1\Sigma$.
\end{theorem}

\section{About computation}
In \cite{HeisenbergHomology} we have used a natural intersection pairing for computing the action of the standard left handed  twists  in the case $n=2$. The intersection pairing can be avoided by using a Fox type calculus which we sketch here. Let us consider the relative ribbon graph depicted in Figure \ref{modelG}. Denote by $\alpha$, $\beta$ the oriented cores of the index $1$ handles as depicted in Figure \ref{genus1}. The arc $A$ is the horizontal part of the boundary.
\begin{figure}
\centering
\includegraphics[scale=0.8]{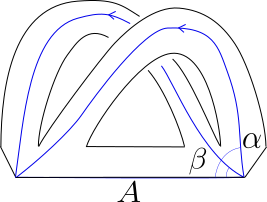}
\caption{Surface $\Sigma_{1,1}$ with tether from the base $2$ points configuration to $\alpha$ and $\beta$.}
\label{genus1}
\end{figure}
Recall from Section \ref{HeisSection} that 
the relative homology with twisted coefficient $H_*(\mathcal{C}_{n}(\Sigma), \mathcal{C}_{n}(\Sigma,A)\setminus T ; W)$
is obtained from the quotient of the complex $S_*(\mathcal{C}_n(\Sigma_{g,m});W)$ by the subcomplex 
$ S_*(\widetilde{\mathcal{C}}_n(\Sigma_{g,m})\setminus \widetilde T)\otimes_{\Z[\Heis(\Sigma_{g,m})] }W$, where $\widetilde T$ is the inverse image of $T$ in the covering space $\widetilde{\mathcal{C}}_n(\Sigma_{g,m})$.

We obtain a basis of the relative Borel-Moore homology with coefficients in the regular representation $\Z[\Heis]$,
$$H_2^{BM}({\mathcal{C}}_2(\Sigma_{1,1}),{\mathcal{C}}_2(\Sigma_{1,1},A);\Z[\Heis])\cong
H_2^{BM}({\mathcal{C}}_2(\Gamma),{\mathcal{C}}_2(\Gamma,A);\Z[\Heis]) \ ,$$
with the open simplices
$w(\alpha)=\mathcal{C}_2(\alpha)$, $w(\beta)=\mathcal{C}_2(\beta)$ and the square $v(\alpha,\beta)=\alpha\times\beta$. Here the lift in the Heisenberg cover is fixed by a path from the base configuration. For $v(\alpha,\beta)$ this tether is depicted in Figure \ref{genus1}.  It connects the first base point to $\alpha$ and the second to $\beta$.
 We connect similarly the simplices,  with a subpath of the previous one for $w(\alpha)$ and a prolongated path for $w(\beta)$.
We denote by $a$ and $b$ closed curves which are parallel to $\alpha\cup A$ and $\beta \cup A$ respectively. We will use the same notation $a$, $b$ for $(0,a)$ and $(0,b)$ in the Heisenberg group $\Heis=\Heis(\Sigma_{g,1})$ and $u=(1,0)\in \Heis$. We are able to compute the matrices for the action of the right-handed Dehn twists $T_a$ and $T_b$ in the given basis.

\begin{proposition} The matrices for the action of the twists $T_a$ and $T_b$ on $H_2^{BM}({\mathcal{C}}_2(\Sigma_{1,1}),{\mathcal{C}}_2(\Sigma_{1,1},A);\Z[\Heis])$ in the above described basis are
$$M_a= \left(\begin{array}{rrr}
 1 &  1 &  -u + 1 \\
0 &  u^{2} a^{ 2 } & 0 \\
0 & a & a
\end{array}\right)$$

$$M_b=\left(\begin{array}{rrr}
 1 & 0 & 0 \\
 -u^{7} a^{ 2 }b^{ -2 } &  1 & ( -u^{4} + u^{3} )a b^{ -1 } \\
 -u^{2} a b^{ -1 } & 0 &  1
\end{array}\right)$$
\end{proposition}

\begin{proof}
We have that the twist $T_a$ fixes $\alpha$ and sends $\beta$ to a curve isotopic to $\beta A \alpha$. Recall that we compose paths from the right. 
We get that $w(\alpha)$ is fixed, $T_a.w(\beta)$ is represented by $w(\beta A \alpha)$ and $T_a.v(\alpha,\beta)$ is represented by $v(\alpha,\beta A \alpha)$. Here the notation $w$ and $v$ describes  Borel-Moore relative cycles which are respectively a $2$-simplex and a square supported on the given curves. We then have to analyse carefully those cycles and decompose them in the basis.

Modulo points in $A$, the simplex $w(\beta A \alpha)$ decomposes into two simplices and a square giving in relative Borel-Moore homology a relation
$$w(\beta A \alpha)=w(\alpha)+ \lambda\,v(\alpha,\beta) + \mu\,w(\beta)\ .$$
The coefficients $\lambda$ and $\mu$ are composed of a sign coming from comparing the orientations and an element in the Heisenberg group coming from the deck transformation. The latter one is obtained by evaluating a braid which starts from the base configuration, goes along $w(\beta A \alpha)$ and come back to the base configuration along the basis cycle. We obtain
$$w(\beta A \alpha)=w(\alpha)+ a\,v(\alpha,\beta) + u^2a^2\,w(\beta)\ .$$
In a similar way, the square $v(\alpha,\beta A \alpha)$ decomposes into two simplices and a square. Here both simplices are up to orientation deck translations of the same basis element. We obtain
$$v(\alpha,\beta A \alpha)= (-u+1)\,w(\alpha)+a\,v(\alpha,\beta) \ .$$
The twist $T_b$ sends $\beta$ to $\beta$ and $\alpha$ to a curve isotopic to $\beta^{-1}A\alpha$. The matrix for the action is obtained with the same procedure.
\end{proof}
\begin{remark}
The computation is a bit tricky. We expect a better setup with an appropriate enhancement of Fox calculus. We checked that our matrices satisfy the braid relations. Here, due to the twisted action, we have to act on matrices by the appropriate automorphism of the Heisenberg group when composing mapping classes. Then checking the braid relation $T_aT_bT_a=T_bT_aT_b$ requires the equality
$$M_a\times(T_a)_\Heis.M_b\times (T_aT_b)_\Heis. M_a=M_b\times(T_b)_\Heis.M_a\times (T_bT_a)_\Heis. M_b\ .$$
Here $(T_a)_\Heis(a)=a$, $(T_a)_\Heis(b)=ba=u^{-2}ab$, $(T_b)_\Heis(a)=b^{-1}a=u^2ab^{-1}$, $(T_b)_\Heis(b)=b$.
We verified this equality by hand and with Sage program. We also checked that the composition corresponding to the twist along the boundary, equal to $(T_aT_b)^6$, produces a matrix which commutes with $M_a$ and $M_b$, with appropriate twisting by $(T_a)_\Heis$, $(T_b)_\Heis$.
\end{remark}

\printindex

\bibliographystyle{unsrt}
\bibliography{biblio}

\end{document}